\documentclass[12pt,a4paper]{amsart}
\usepackage{tikz,xcolor}
\newtheorem{theo+}              {Theorem}           [section]
\newtheorem{prop+}  [theo+]     {Proposition}
\newtheorem{coro+}  [theo+]     {Corollary}
\newtheorem{lemm+}  [theo+]     {Lemma}
\newtheorem{exam+}  [theo+]     {Example}
\newtheorem{rema+}  [theo+]     {Remark}
\newtheorem{defi+}  [theo+]     {Definition}
\def \r{\mbox{${\mathbb R}$}}

\newenvironment{theorem}{\begin{theo+}}{\end{theo+}}
\newenvironment{proposition}{\begin{prop+}}{\end{prop+}}
\newenvironment{corollary}{\begin{coro+}}{\end{coro+}}
\newenvironment{lemma}{\begin{lemm+}}{\end{lemm+}}
\newenvironment{definition}{\begin{defi+}}{\end{defi+}}

\usepackage{amsthm}
\theoremstyle{plain} \theoremstyle{remark}
\newtheorem{remark}{Remark}
\newtheorem{example}{Example}

\def \r{\mbox{${\mathbb R}$}}

\def\M{$M_{\sigma}^m(o)$\;}

\title{Bi-Laplacian and bi-eigenvalues on some model spaces}
\author{Ye-Lin Ou }

\address{Department of Mathematics,\newline\indent
Texas A $\&$ M University-Commerce,\newline\indent Commerce, TX
75429, U S A.\newline\indent E-mail:yelin.ou@tamuc.edu}
\begin{document}

\title[Biharmonic functions and bi-eigenfunctions on some model spaces]{Biharmonic functions and bi-eigenfunctions on some model spaces}

\date{10/6/2024}
\maketitle

\maketitle
\section*{Abstract}
\begin{quote}
{\footnotesize 
} In this paper, we first give a convenient formula for bi-Laplacian on  a sphere and the complete description of its eigenvalues, buckling eigenvalues, and their corresponding eigenfunctions. We then show that the radial (or rotationally symmetric) solutions for biharmonic equation on the model space $(  \r^+\times S^{m-1},  dr^2 + \sigma^2(r)\, g^{S^{m-1}})$ can be given  by an integral formula. We also prove that the model space always admits proper biharmonic functions as the products of any eigenfunctions of the factor sphere with certain radial functions. Many explicit examples of proper biharmonic functions on space forms are given. Finally, we give a complete classification of proper biharmonic functions with positive Laplacian on the punctured Euclidean space.
\end{quote}

\section{Introduction}

Our sign convention for the Laplace operator acting on functions on a Riemannian manifold $(M, g)$ is $\Delta =div (\nabla f)$, and the bi-Laplace operator is denoted by $\Delta^2 f=\Delta (\Delta f)$.\\

Biharmonic  functions are solutions of bi-Laplace equations $\Delta^2 u=0$.\\

      Biharmonic functions had been found to have important applications in mathematics and physics ever since the work of Airy (1862) and Maxwell on the solution to the problem of the stress distribution on a thin plate (see e.g, \cite{TG}). Recall that to solve this stress distribution problem one is lead to the equilibrium  equation with some boundary and compatibility conditions, which can be put as 
\begin{align}
\begin{cases}
\frac{\partial \sigma_{x}}{\partial x}+\frac{\partial
\tau_{xy}}{\partial y}=0,\\
 \frac{\partial
\sigma_{y}}{\partial y}+\frac{\partial \tau_{xy}}{\partial x}\, \rho g=0,\\
\Delta (\sigma_{x}+\sigma_{y})=0,
\end{cases}
\end{align}
where $\sigma_{x}, \sigma_{y}, \tau_{xy}$ are components of stress at a point of the plate.\\

G. B. Airy (1862) solved the equation by introducing the so-called stress function $\phi(x, y)$ which satisfies
\begin{equation}
 \sigma_{x}=\frac{\partial^2 \phi}{\partial y^2}-\rho g y,\:\:
\sigma_{y}=\frac{\partial^2 \phi}{\partial x^2}-\rho g y,\;\;
\tau_{xy}=-\frac{\partial^2 \phi}{\partial x\partial y}.
\end{equation}

It was Maxwell  who discovered that Airy's stress functions $\phi(x,y)$ are  biharmonic functions.\\

A series of important applications of biharmonic functions  in classifying Riemannian manifolds were also made  by  Sario, Nakai,  Wang, and  Chung during the late  1970s. Their classification theory, which generalizes the ideas from uniformization theorem and classification of Riemann surfaces, helps  to classify Riemannian manifolds by the existence or nonexistence of  biharmonic functions with  various boundedness conditions (such as positive, bounded, Dirichlet finite, or both bounded and Dirichlet finite, etc).

Related to bi-Laplacian there is a lot of recent work, see e.g., \cite{CY06, WX1, CY11, CC, Xi, WX2, ZZ}) and the references therein, which study the local or the global  bi-eigenfunctions and buckling functions on a Riemannian manifold.

\begin{align}\notag
  The\;Clamped\; plate\; problem:&\;\begin{cases}
\Delta^2 f=\mu f\;\;on \; M\\
f|_{\partial M}=\frac{\partial f}{\partial n}|_{\partial M}=0,\\
\end{cases}\\\notag
  The\;buckling\; eigenvalue\; problem:&\; \begin{cases}
\Delta^2 f=-\nu \Delta f\;\;on \; M\\
f|_{\partial M}=\frac{\partial f}{\partial n}|_{\partial M}=0.
\end{cases}
\end{align}

Finally, we note that biharmonic functions are special cases of biharmonic maps between Riemannian manifolds, the geometric study of the latter has lead to many fruitful results and interesting links among differential geometry, differential equations, geometric analysis, and calculus of variations. See a recent book \cite{OC} and the vast references therein some recent progress on the study of biharmonic maps.

In general, computing bi-Laplacian and solving the 4th order bi-Laplace  or bi-eigenvalue problem is not an easy job. For this reason, we have not seen many examples of proper biharmonic functions on a sphere or other model spaces.

In this paper, we first give a convenient formula for bi-Laplacian on  a sphere and the complete description of its eigenvalues, buckling eigenvalues, and their corresponding eigenfunctions. We then show that the radial (or rotationally symmetric) solutions for biharmonic equation on the model space $(  \r^+\times S^{m-1},  dr^2 + \sigma^2(r)\, g^{S^{m-1}})$ can be given  by an integral formula. We also prove that the model space always admits proper biharmonic functions as the products of any eigenfunctions of the factor sphere with certain radial functions. Many explicit examples of proper biharmonic functions on space forms are given. Finally, we give a complete classification of proper biharmonic functions with positive Laplacian on the punctured Euclidean space.

\section{Bi-Laplacian, Bi-eigenfunctions and buckling eigenfunctions on a sphere}

In this section, we  derive a formula for the bi-Laplacian on functions on a sphere using the Euclidean bi-Laplacian and give a complete description of all bi-eigenvalues, buckling eigenvalues and the corresponding bi-eigenfunctions and buckling eigenfunctions on  the sphere $S^m$.

{\bf 1.1 A formula for bi-Laplacian on functions on the sphere $S^m$}\\

Let ${\bf i}: S^{m}\to \r^{m+1}$ denote the standard isometric embedding of the unit sphere into Euclidean space. It is well known (see e.g., \cite{OC}, Lemma 9.6) that the Euclidean Laplacian and the Laplacian on a sphere are related by
\begin{eqnarray}\label{LaS}
\Delta_{S^m}(F\circ{\bf i}) =\left(\Delta_{\r^{m+1}} F-\frac{\partial^2F}{\partial r^2}-m\frac{\partial F}{\partial r}\right)\circ{\bf i}.
\end{eqnarray}
where $\frac{\partial}{\partial r}$ denotes the directional derivative in the radial direction.

Using the fact that the radial projection $\pi: \r^{m+1}\to S^{m}, \pi(x)=\frac{x}{|x|}$ is a harmonic morphism (a horizontally weakly conformal  harmonic map) we derive the following  convenient formula to compute the bi-Laplacian on a sphere using the data and the operators on the ambient Euclidean space.

\begin{theorem}
For any function  $f:S^m\supseteq U\to \r$, we have
\begin{align}\label{RS1}
\Delta^2_{S^m} f= \left(\Delta^2_{\r^{m+1}}f\left( \frac{x}{|x|}\right)\right)_{S^m}+2(m-3)\Delta_{S^m} f,\; or\\\label{RS2}
\Delta^2_{S^m} f= \left(\Delta^2_{\r^{m+1}}f\left( \frac{x}{|x|}\right)+2(m-3)\Delta_{\r^{m+1}} f\left( \frac{x}{|x|}\right)\right)_{S^m}.
\end{align}
\end{theorem}

\begin{proof}
It is well known (see e.g. \cite{BW}) that the radial projection $\pi:\r^{m+1}\to S^m$ with $\pi(x)=\frac{x}{|x|}$ is a horizontally homothetic harmonic morphism with dilation $\lambda=\frac{1}{|x|}$. It follows that for any function $f:S^m\supseteq U\to \r$, we have
\begin{equation}\label{RP}
\Delta_{\r^{m+1}}(f\circ \pi)=\lambda^2(\Delta_{S^m} f)\circ \pi, \; \forall\; x\in \pi^{-1}(U)\subseteq \r^{m+1}\setminus\{0\},\;i.e.,
\end{equation}
\begin{align}\notag
\Delta_{\r^{m+1}}f(\frac{x}{|x|})=& \frac{1}{|x|^2}(\Delta_{S^m} f)\,(\frac{x}{|x|}), \; \forall\; x\in  \pi^{-1}(U)\subseteq \r^{m+1}\setminus\{0\},\; {\rm or}\\\label{D28}
(\Delta_{S^m} f)\circ \pi  =&  |x|^2\,\Delta_{\r^{m+1}}\left(f(\frac{x}{|x|})\right),\;  \forall\; x\in  \pi^{-1}(U)\subseteq \r^{m+1}\setminus\{0\}.
\end{align}

Evaluating the above equation on $S^m$ we have
\begin{align}\label{DSG}
\Delta_{S^m} f =\left(\Delta_{\r^{m+1}}f(\frac{x}{|x|})\right)_{S^m}.
\end{align}

Applying $\Delta_{\r^{m+1}}$ to both sides of (\ref{RP}), using the product rule for Laplace operator,  and a straightforward computation yields
\begin{align}\notag
\Delta^2_{\r^{m+1}}(f\circ \pi) &=\Delta_{\r^{m+1}}(\Delta_{\r^{m+1}}(f\circ \pi))=\Delta_{\r^{m+1}}(\lambda^2(\Delta_{S^m}f)\circ \pi)\\\notag
&=\lambda^4(\Delta^2_{S^m} f)\circ\pi +2\langle \nabla\lambda^2, \nabla((\Delta_{S^m}f)\circ \pi) \rangle+ (\Delta_{\r^{m+1}}\lambda^2)(\Delta_{S^m} f)\circ\pi\\\label{2D}
&=\lambda^4(\Delta^2_{S^m} f)\circ\pi+2(\Delta_{S^m}f)_{\alpha}\circ\pi\; \langle \nabla\lambda^2, \nabla\pi^{\alpha}\rangle+(\Delta_{\r^{m+1}}\lambda^2)(\Delta_{S^m} f)\circ\pi.
\end{align}
Using the fact that $\pi$ is horizontally homothetic and hence $\langle \nabla\lambda^2, \nabla\pi^{\alpha}\rangle=0, \forall\, \alpha=1, 2, \cdots, m$, together with  
$\lambda^2=\frac{1}{|x|^2},\;\;\;\Delta_{\r^{m+1}}\lambda^2=-\frac{2(m-3)}{|x|^4}$, we have

\begin{align}\notag
\Delta^2_{\r^{m+1}}(f\circ \pi)=\frac{1}{|x|^4}(\Delta^2_{S^m} f)\circ\pi -\frac{2(m-3)}{|x|^4}(\Delta_{S^m} f)\circ\pi, or\\\label{KE}
|x|^4 \Delta^2_{\r^{m+1}}f\left( \frac{x}{|x|}\right)=(\Delta^2_{S^m} f)\left( \frac{x}{|x|}\right)-2(m-3)(\Delta_{S^m} f)\left( \frac{x}{|x|}\right).
\end{align}

Evaluate both sides of (\ref{KE}) on $x\in S^m$ (i.e., $|x|=1$) gives (\ref{RS1}). Using (\ref{RS1}) and (\ref{DSG}) we have

\begin{align}\notag
\left(\Delta^2_{\r^{m+1}}f\left( \frac{x}{|x|}\right)\right)_{S^m}=\Delta^2_{S^m} f-2(m-3)\left(\Delta_{\r^{m+1}}f\left( \frac{x}{|x|}\right)\right)_{S^m},
\end{align}
from which we obtain (\ref{RS2}).

\end{proof}

As immediate consequences, we have
\begin{corollary}
(i) For any function $f:S^3 \supseteq U\to \r$, we have
\begin{align}\label{q3}
\Delta^2_{S^3} f &= \left(\Delta^2_{\r^{4}}f\left( \frac{x}{|x|}\right)\right)_{S^3}.
\end{align}
(ii) A function $f:S^2 \supseteq U\to \r$ is a buckling eigenfunction of eigenvalue $\mu=2$, i.e., $\Delta^2_{S^2} f= -2 \Delta_{S^2} f$ if and only if  $\left(\Delta^2_{\r^{3}}f\left( \frac{x}{|x|}\right)\right)_{S^2}=0$.

\end{corollary}

\begin{example}
The restriction of $f(x)=\ln (1-x_3), \;x\in \r^3$  is a proper biharmonic function on $S^2$. In fact,  in this case, we have $(f\circ \pi)(x)=f(\frac{x}{|x|})=\ln (1-\frac{x_3}{|x|})$. A straightforward computation yields
\begin{align}
\Delta_{\r^{3}}\left(f(\frac{x}{|x|})\right) & =- \frac{1}{|x|^2},\;\;\;\Delta_{\r^{3}}^2\left(f(\frac{x}{|x|})\right)  =-\frac{2}{|x|^4}.
\end{align}

Substituting these into (\ref{DSG}) and (\ref{RS2}) we have, respectively
\begin{align}
\Delta_{S^2} f & =\left(\Delta_{\r^{3}}\left(f(\frac{x}{|x|})\right) \right)_{|x|=1}=1\ne 0,\\
\Delta_{S^2}^2 f & = \left(\Delta^2_{\r^{3}}f\left( \frac{x}{|x|}\right)-2\Delta_{\r^{3}} f\left( \frac{x}{|x|}\right)\right)_{S^2} \equiv 0.
\end{align}
Thus, $f(x)=\ln (1-x_3), \;x\in S^2$ is a proper biharmonic function  on the sphere $S^2\setminus\{N\}$.
\end{example}

\begin{example}
The restriction of  $f(x)=(a x_1+bx_2+cx_3)/\sqrt{1-x_4^2}$ to $S^3\setminus\{N, S\}$ is a family of proper biharmonic function on 3-sphere.\\
 
 In fact, $f(\frac{x}{|x|})=(a x_1+bx_2+cx_3)/\sqrt{\sum_{i=1}^3x_i^2}$, and a straightforward computation gives 
 \begin{align}
  \Delta_{\r^{4}}\left(f(\frac{x}{|x|})\right)&=\Delta_{\r^{4}}\left( \frac{a x_1+bx_2+cx_3}{\sqrt{x_1^2+x_2^2+x_3^2}}\right)=\frac{-2(a x_1+bx_2+cx_3)}{(x_1^2+x_2^2+x_3^2)^{3/2}},\\\notag
   \Delta^2_{\r^{4}}\left(f(\frac{x}{|x|})\right)&=\Delta_{\r^{4}}\left( \frac{-2(a x_1+bx_2+cx_3)}{(x_1^2+x_2^2+x_3^2)^{3/2}}\right)=0.
 \end{align}
 It follows from  (\ref{q3}) that $\Delta^2_{S^3} f = \left(\Delta^2_{\r^{4}}f\left( \frac{x}{|x|}\right)\right)_{S^3}=0$.
 \end{example}

\begin{corollary}\label{Pk}
$f=F|_{S^m}$ is the  restriction of a homogeneous polynomial $F$ of degree $k$ on $\r^{m+1}$, then we have
\begin{align}\label{D2S}
\Delta^2_{S^m} f & =  k (k+2) (k -1+m) (k+m-3)\,f(x) +2(m-3)\Delta_{S^m} f\\\notag
&+\left(\Delta^2_{\r^{m+1}}F(x) -2k (k+m-3)\, \Delta_{\r^{m+1}} F(x)\right)_{S^m}, \;\;or\\\label{D2S2}
\Delta^2_{S^m} f  = &\;  k^2 (k -1+m)^2\,f  \\\notag
 &+\left(\Delta^2_{\r^{m+1}} F(x)  -2[k^2+(k-1)(m-3)] \Delta_{\r^{m+1}} F(x)\right)_{S^m}.
 \end{align}
\end{corollary}

\begin{proof}
It is easily checked (see also \cite{ABR}, Lemma 4.4) that if $f:S^m \to \r$ with $f=F|_{S^m}$ is a restriction of a homogeneous polynomial  $F$ of degree $k$, then we have
  \begin{align}\label{X1}
  \nabla |x|^{\alpha}& =\alpha |x|^{\alpha-2}\,x,\\\label{X2}
 \Delta_{\r^{m+1}}\, |x|^{\alpha}& = \alpha (\alpha -1+m)\,|x|^{\alpha-2},\\\label{X3}
 \langle x, \nabla F(x)\rangle & =k F(x).
  \end{align}
 Since $f=F|_{S^m}$ is a restriction of a homogeneous polynomial  $F$ of degree $k$, we have  $f (\frac{x}{|x|})=F(\frac{x}{|x|})=|x|^{-k}F(x)$. A straightforward computation using  (\ref{X1})-(\ref{X3}) gives
  \begin{align}\notag
\Delta_{\r^{m+1}}f\left(\frac{x}{|x|}\right)& = \Delta_{\r^{m+1}} (|x|^{-k}F(x)) \\\notag
 & = (\Delta_{\r^{m+1}} \,|x|^{-k})\,F(x)+ |x|^{-k} \Delta_{\r^{m+1}} F(x)+2\langle \nabla |x|^{-k}, \nabla F(x)\rangle\\\label{X4}
     & =  -k (k -1+m)\,|x|^{-k-2}\,F(x)+ |x|^{-k} \Delta_{\r^{m+1}} F(x).
 \end{align}

Evaluate both sides of the above on $S^m$ we have
\begin{align}\label{DS}
(\Delta_{S^m}F|_{S^m})(x)&=\left( \Delta_{\r^{m+1}} |x|^{-k}F|_{S^m}(x)\right)_{|x|=1}\\\notag
&= -k (k -1+m)\,F|_{S^m}(x)+  (\Delta_{\r^{m+1}} F(x))|_{S^m}.
\end{align}

Applying $\Delta_{\r^{m+1}}$ to both sides of  (\ref{X4}) and a further calculation using (\ref{X1})-(\ref{X3})  yields
\begin{align}\label{X5}
& \Delta^2_{\r^{m+1}}f\left( \frac{x}{|x|}\right)= \Delta_{\r^{m+1}}[ -k (k -1+m)\,|x|^{-k-2}\,F(x)+ |x|^{-k} \Delta_{\r^{m+1}} F(x)]\\\notag
&=  -k (k -1+m)\,\Delta_{\r^{m+1}}[|x|^{-k-2}\,F(x)]+ \Delta_{\r^{m+1}}[|x|^{-k} \Delta_{\r^{m+1}} F(x)]\\\notag
&=  k (k+2) (k -1+m) (k+m-3)\;|x|^{-k-4}\,F(x) -2k (k+m-3)\,|x|^{-k-2}\, \Delta_{\r^{m+1}} F(x)\\\notag
&+|x|^{-k} \Delta^2_{\r^{m+1}}F(x).
\end{align}

Evaluating both sides of (\ref{X5}) on $S^m$ and substituting the result into (\ref{RS1}) gives (\ref{D2S}). Finally,  (\ref{D2S2}) is obtained by substituting (\ref{DS}) into (\ref{D2S}).
\end{proof}

\begin{remark} 
Note that Formulas (\ref{DSG})  and (\ref{DS}) were proved in \cite{Sh} in a different way. 
\end{remark}

{\bf 1.2 Bi-eigenvalues, buckling eigenvalues, and their eigenfunctions on the sphere $S^m$}\\

\begin{definition}
A function $f\in C^{\infty}(M)$ is called an eigenfunction of the Laplace operator with eigenvalue $\lambda$ if it solves the PDE
\begin{align}\label{Egen}
\Delta f=-\lambda f.
\end{align}
The set of all eigenvalues (called the spectrum) of $\Delta$ is denoted by $Spec_{\Delta}(M)$. The eigenspace corresponding to the eigenvalue $\lambda_i$ is denoted by 

\begin{align}\notag
V(\lambda_i)=\{f\in C^{\infty}(M)\;|\;  \Delta f=-\lambda_i f\}.
\end{align}

A function $f\in C^{\infty}(M)$ is called an eigenfunction of the bi-Laplace operator with eigenvalue $\mu$ if it solves the PDE
\begin{align}\label{2-E}
\Delta^2 f=\mu f.
\end{align}
For convenience,  eigenfunctions and eigenvalues of the bi-Laplace operator are also called {\bf bi-eigenfunctions}  and {\bf bi-eigenvalue} of the Laplacian respectively.\\

The set of all bi-eigenvalues of $\Delta$ is denoted by $Spec_{\Delta^2}(M)$. The bi-eigenspace corresponding to the bi-eigenvalue $\mu_i$ is denoted by 
\begin{align}\notag
V^2(\mu_i)=\{f\in C^{\infty}(M)\;|\;  \Delta^2 f=\mu_i f\}.
\end{align}

A function $f\in C^{\infty}(M)$ is called a {\bf buckling eigenfunction} with a {\bf buckling eigenvalue} $\nu$ if it solves the PDE
\begin{align}\label{2-ED}
\Delta^2 f=-\nu \Delta f.
\end{align}
The set of all buckling-eigenvalues  is denoted by $Spec_{\Delta^b}(M)$. The buckling-eigenspace corresponding to the buckling-eigenvalue $\nu_i$ is denoted by 
\begin{align}\notag
V^b(\nu_i)=\{f\in C^{\infty}(M)\;|\;  \Delta^2 f=- \nu_i \Delta f\}.
\end{align}
\end{definition}

It is well known (see e.g., \cite{Xi}) that all of the above three types of eigenvalues are countable so we can list them as

\begin{align}\notag
0=& \lambda_0 < \lambda_1<\lambda_2\le \lambda_3\le \cdots +\infty\hskip2cm  (\rm Laplacian\; eigenvalues)\\\notag
0=& \mu_0 < \mu_1 < \mu_2\le \mu_3\le \cdots +\infty\hskip2cm (\rm bi-eigenvalues)\\\notag
0=& \nu_0 < \nu_1 < \nu_2\le \nu_3\le \cdots +\infty\hskip2cm (\rm Buckling\; eigenvalues)\\\notag
\end{align}

\begin{remark}\label{P2}
One can easily check the following\\
(i) The square of any eigenvalue of the Laplacian $\Delta$ is an eigenvalue of  the bi-Laplacian $\Delta^2$, and any eigenfunction is automatically a bi-eigenfunction, i.e.,
\begin{align}
  Spec_{\Delta^2} (M) \supset \{ \lambda^2\;|\; \forall\; \lambda \in Spec_{\Delta} (M) \},\; and\;\; V^2(\lambda^2)\supset V(\lambda).
\end{align}
(ii) The set of all buckling eigenvalues on $(M, g)$ is the same as the set of all eigenvalues of the Laplacian $\Delta$,  and a buckling eigenfunction is precisely a function whose Laplacian is an eigenfunction of $\Delta$, i.e.,
\begin{align}
  Spec_{\Delta^b} (M) = Spec_{\Delta} (M),\; and\;\; V^b(\lambda)=\{ f\in C^{\infty}(M)\;|\; \Delta f \in V(\lambda)\}.
\end{align}
\end{remark}

For eigenvalues and the corresponding eigenspaces of the Laplacian on a sphere, we have the following well-known results.

\begin{proposition}\label{H}(see e.g., \cite{Sh} Proposition 22.2.)
The restriction  to $S^{m}$ of any harmonic homogeneous polynomial of degree $k$ on $\r^{m+1}$ is an eigenfunction of the Laplacian $\Delta_{S^m}$  with eigenvalue $\lambda_k=k(m+k-1)$. Conversely, any eigenfunction of $\Delta_{S^m}$  on the sphere $S^m$ is a restriction of a homogeneous polynomial of degree $k$.
\end{proposition}

Now  we are ready to prove the following theorem which gives complete descriptions of bi-eigenvalues and buckling eigenvalues and the corresponding eigenspaces of the sphere.
\begin{theorem}\label{EV}
(I) All eigenvalues of bi-Laplacian on $S^m$ come from a square of an eigenvalue of the Laplacian, and bi-eigenfunction on $S^m$ are precisely the restrictions of harmonic homogenous polynomial in $\r^{m+1}$. This can described as
\begin{align}
  Spec_{\Delta^2} (S^m) =&\; \{ \lambda_k^2\;|\;  \lambda_k=k(m+k-1),\; \forall\; k \in \mathbb{N} \},\; and\;\\
  &\; V^2(\lambda_k^2)=\mathcal{H}^k,
\end{align}
where $\mathcal{H}^k$ denotes the set of harmonic homogeneous polynomials of degree $k$. 
In particular, the first nonzero eigenvalue of the bi-Laplacian on $S^m$ is $\mu_1=m^2$.\\

(II) The set of all buckling eigenvalues on $S^m$ is the same as the set of all eigenvalues of the Laplacian $\Delta$ on $S^m$, and buckling eigenfunctions on $S^m$ are precisely the restrictions of polynomial of the form $h_k+c|x|^2$. This can be described as
\begin{align}\notag
 Spec_{\Delta^b} (S^m) =&\; Spec_{\Delta} (S^m)= \{ \lambda_k=k(m+k-1)\;|\; \forall\; k \in \mathbb{N}  \}, and\\\notag &\;V^b(\lambda_k)=\mathcal{H}^k\oplus \{c |x|^2\}.
\end{align}
In particular, the first nonzero buckling eigenvalue of  $S^m$ is $\mu_1=m$.
\end{theorem}
\begin{proof}

Recall that for a closed manifold $(M, g)$, $C^{\infty}(M)$ is an inner product space with the inner product defined by
\begin{align}\label{HI}
\langle \alpha, \beta\rangle= \int_M \alpha \beta dv_g
\end{align}
and the Laplace operator $\Delta: C^{\infty}(M)\to C^{\infty}(M)$ is a self-adjoint operator, so we have
\begin{align}
\langle \Delta \alpha, \beta\rangle= \langle \alpha, \Delta \beta\rangle, \;\;\forall\; \alpha, \beta\in \; C^{\infty}(M). 
\end{align}

Note also that  the space $L^2(S^m)$ of all real-valued square-integrable functions on $S^m$ is a Hilbert space with the inner product also defined by (\ref{HI}).

Let $\lambda_k$ be an eigenvalue of the Laplacian on $S^m$, i.e., $\Delta_{S^m} f=-\lambda f$ for some function on $S^m$. Then, as in Remark \ref{P2}, $\lambda_k^2$ is a bi-eigenvalue. Conversely, if $\Delta^2_{S^m}f=\mu f$ for a constant $\mu$, we are to prove that $\mu=\lambda_k^2$ for one and only one $k\in \mathbb{N}$. It is well known (see e.g., \cite{ABR}, Theorem 5.12)  that $L^2(S^m)=\oplus_{i=0}^{\infty} \mathcal{H}^i$ as the direct sum of spherical harmonic polynomials of homogeneous degree $k$.
For any smooth function $f$ on $S^m$, we have $f=\sum_{i=0}^{\infty}h_i$ in $L^2(S^m)$. For $k\in \mathbb{N}$, we have
\begin{align}
|| h_k||_{L^2(S^m)} & =\langle f, h_k\rangle_{L^2(S^m)}=\frac{1}{\lambda_k^2}\langle f,  \Delta_{S^m}^2h_k\rangle_{L^2(S^m)}=\frac{1}{\lambda_k^2}\langle \Delta_{S^m}^2f,  h_k\rangle_{L^2(S^m)}\\\notag
& =\frac{1}{\lambda_k^2}\langle \mu f,  h_k\rangle_{L^2(S^m)}=\frac{\mu}{\lambda_k^2}\langle f,  h_k\rangle_{L^2(S^m)}=\\\notag
& = \frac{\mu}{\lambda_k^2}\langle f, h_k\rangle_{L^2(S^m)}=\frac{\mu}{\lambda_k^2} || h_k||_{L^2(S^m)}.
\end{align}
It follows that 
\begin{equation}
(1-\frac{\mu}{\lambda_k^2}) || h_k||_{L^2(S^m)}=0.
\end{equation}

Case (i):  if $\mu\ne \lambda^2_k$ for any $k\ge 1$, then $h_k=0, \forall\; k\ge1$ in $L^2(S^m)$ and hence $f=h_0$ is a constant function. In this case, we have $\Delta^2_{S^m} f=0= \mu f$ and $\mu=0$.

Case (ii): if $\mu=\lambda_k^2$ for some $k\in \mathbb{N}$, i.e., $\Delta_{S^m}^2 f=\mu f= \lambda^2_k f$, we will show that $f$ is orthogonal to $\mathcal{H}^l$ for any $l\ge 1$ and $l \ne k$ in  $L^2(S^m)$. In fact, for any $l\ge 1$ and any $h_l\in \mathcal{H}^l$ we have
\begin{align}
\langle h_l,  f\rangle_{L^2(S^m)} & =\frac{1}{\lambda_l^2}\langle \Delta_{S^m}^2 h_l,   f\rangle_{L^2(S^m)}=\frac{1}{\lambda_l^2}\langle  h_l,  \Delta_{S^m}^2 f\rangle_{L^2(S^m)}\\\notag
&  = \frac{1}{\lambda_l^2}\langle  h_l,   \lambda_k^2 f\rangle_{L^2(S^m)}= \frac{\lambda_k^2}{\lambda_l^2}\langle  h_l,   f\rangle_{L^2(S^m)}.
\end{align}
From this and the fact that $ \frac{\lambda_k^2}{\lambda_l^2}\ne 1$ we conclude that $f$ is orthogonal to $\mathcal{H}^l$ for any $l\ge 1$ and any $l\ne k$ in  $L^2(S^m)$, so we conclude that $f$ must be of the form $f=h_0+h_k\in \mathcal{H}^0\oplus \mathcal{H}^k$. In this case, $\Delta_{S^m}^2 f=\lambda^2_k f$ if and only if $h_0=0$. Thus, any bi-eigenfunction corresponding to the bi-eigenvalue $\lambda_k^2$ is a restriction of a harmonic homogeneous polynomial of degree $k$. This completes the proof of Statement (I).\\

For Statement (II), by Remark \ref{P2}, we only need to prove that the buckling eigenspace corresponding to the buckling eigenvalue $\lambda_k$ is $V^b(\lambda_k)=\mathcal{H}^k\oplus\{c|x|^2\;|\; c\in \r\}$. In fact, if $f\in C^{\infty}(S^m)$  is a buckling eigenfunction corresponding to eigenvalue $\lambda_k$ such that $\Delta_{S^m}^2f=-\lambda_k\,\Delta_{S^m} f$, then, rewriting it as $\Delta_{S^m}(\Delta_{S^m} f)=-\lambda_k\,(\Delta_{S^m} f)$, we see that $\Delta_{S^m} f$ is an eigenfunction of $\Delta_{S^m} $ with eigenvalue $\lambda_k$. By Proposition \ref{H}, we have  
\begin{align} \label{21}
\Delta_{S^m} f=q_k|_{S^m}\; for\; some \; q_k\in \mathcal{H}^k.
\end{align}
 Using the unique representation of $f$ in $L^2(S^m)=\oplus_{i=0}^{\infty}\mathcal{H}^i$ we have $f=\sum_{i=0}^{\infty} h_i$. We compute the ith components of $f$ for $i\ge 1$ to have
\begin{align}\label{22}
\langle f, h_i\rangle_{L^2(S^m)} & =-\frac{1}{\lambda_i}\langle f,  \Delta_{S^m} h_i\rangle_{L^2(S^m)}=- \frac{1}{\lambda_i}\langle \Delta_{S^m} f,  h_i\rangle_{L^2(S^m)}\\\notag
& =- \frac{1}{\lambda_i}\langle q_k,  h_i\rangle_{L^2(S^m)}=0,\;\forall\; i\ge 1,and\; i\ne k,
\end{align}
where the 3rd ``$=$" was obtained by using (\ref{21}) and the last ``$=$" holds because $ (q_k, h_i) \in \mathcal{H}^k\oplus \mathcal{H}^i$ and hence $\langle q_k,  h_i\rangle_{L^2(S^m)}=0,  \forall \, i\ge 1,\;{\rm and}\;i\ne k$. It follows from (\ref{22}) that $h_i=0$ for all $i\ne 0, k$ and hence $f=\sum_{i=0}^{\infty} h_i=(h_k+h_0)|_{S^m}$ or $f=(h_k+c|x|^2)|_{S^m}$ for $h_0=c$, which completes the proof of the theorem.
\end{proof}

\begin{corollary}
The restriction of any polynomial of homogeneous degree $2$ to $S^m$ is a buckling eigenfunction (i.e., solution of $\Delta_{S^m}^2f =-\mu \Delta_{S^m} f$) on sphere with $\mu=2(m+1)$. 
\end{corollary}

\begin{proof}
It is well known (see e.g., \cite{ABR}) any homogeneous polynomial of degree $2$ can be written as $p_2(x)=h_2(x)+|x|^2h_0$. From this and Statement (II) of Theorem \ref{EV} we obtain the corollary.
\end{proof}

\begin{remark}
Although the buckling eigenvalue $\mu=2(m+1)$ is the second eigenvalue of the Laplacian $\Delta_{S^m}$,  the buckling eigenspace $V^b({\mu})$ corresponds to $\mu=2(m+1)$ has one dimension bigger than the eigenspace $V({\mu})$ of the Laplace operator. The buckling eigenspace  includes all homogeneous polynomials of degree 2 whilst the latter consists of only all {\bf harmonic} homogeneous polynomials of degree 2.
\end{remark}

For any integer $k\ge 2$, the $k$-Laplacian is define to be $\Delta^k f =\Delta (\Delta^{k-1})$. We define $\mu$ to be an eigenvalue of the  $k$-Laplacian if there is a nonconstant function $f$ such that  $\Delta^k f =(-1)^k\mu \,f$. Then, a similar proof gives the following description of all eigenvalues of the $k$-Laplacian on the sphere.
\begin{corollary}
The spectrum of the $k$-Laplacian on the sphere $S^m$ is 
\begin{align}\notag
Spec_{\Delta^k}(S^m)=\{\lambda_i^k\;|\; \lambda_i=i(m+i-1), \forall\; i\in \mathbb{N}\}. 
\end{align}
In particular, the first nonzero eigenvalue of $\Delta^k$  on $S^m$ is $ m^k$.
\end{corollary}

Many recent work (see e.g., \cite{CY06, WX1, CC, Xi, WX2, ZZ}) study bi-eigenvalue and buckling bi-eigenvalue  problems on a domain of a Riemannian manifold with boundary conditions like 
\begin{align}\notag
  The\;Clamped\; plate\; problem:&\;\begin{cases}
\Delta^2 f=\mu f\;\;on \; M\\
f|_{\partial M}=\frac{\partial f}{\partial n}|_{\partial M}=0,\\
\end{cases}\\\notag
  The\;buckling\; eigenvalue\; problem:&\; \begin{cases}
\Delta^2 f=-\nu \Delta f\;\;on \; M\\
f|_{\partial M}=\frac{\partial f}{\partial n}|_{\partial M}=0.
\end{cases}
\end{align}

\section{Radial biharmonic functions on model space $M^m_{\sigma}(o)$}

It is well known that a model space refers to the Riemannian manifold given by the warped product
\begin{equation}\notag
M_{\sigma}^m(o) = \left (  (0,+\infty)\times S^{m-1},  dr^2 + \sigma^2(r)\, g^{S^{m-1}}\right ),
\end{equation}
where $ (\,S^{m-1},\, g^{S^{m-1}} \, )$ is the standard $(m-1)$-dimensional unit sphere, and  $\sigma:(0,\infty)\to (0,\infty) $ is a smooth function.\\

This model space  plays a very important role in mathematics and physics, and it includes all three space forms $\r^m, S^m$, and $ H^m$ as special cases corresponding to $\sigma(r) =r,\; \sin r$, and $ \sinh r$ respectively.\\

In this section we will show that the biharmonic equation for radial (or rotationally symmetric) solutions on the model space  $M^m_{\sigma}(o)=(\mathbb{R}^+\times S^{m-1}, {\rm d}r^2+\sigma^2(r)g^{S^{m-1}})$ is completely solvable. We give an integral formula for the solutions which  are used to produce families of infinitely many proper biharmonic functions on space forms.

\begin{lemma}\label{MT1}
The Laplacian and the bi-Laplacian operators on the model space $M^m_{\sigma}(o)=(\mathbb{R}^+\times S^{m-1}, {\rm d}r^2+\sigma^2(r)g^{S^{m-1}})$ can be described as
\begin{align}\label{LM}
\Delta_{\sigma} f= & f_{rr} +(m-1)\frac{\sigma'}{\sigma}\, f_r+ \frac{1}{\sigma^2}\,\Delta_{S^{m-1}} f\\\label{2LM}
\Delta_{\sigma}^2 f= & (\Delta_{\sigma} f)_{rr} +(m-1)\frac{\sigma'}{\sigma}\, (\Delta_{\sigma} f)_r+ \frac{1}{\sigma^2}\,\Delta_{S^{m-1}} (\Delta_{\sigma} f).
\end{align}

In particular, for a radial function $f(r, \theta)=f(r)$, we have 
\begin{align}\label{2HF}
\Delta_{\sigma} f= & f'' +(m-1)\frac{\sigma'}{\sigma}\, f'\\\label{LMr}
\Delta_{\sigma}^2 f= & (\Delta_{\sigma} f)''+(m-1)\frac{\sigma'}{\sigma}(\Delta_{\sigma} f)'.
\end{align}
\end{lemma}
\begin{proof}
Let $h$ denote the standard metric on the sphere  $S^{m-1}$ and $g={\rm d}r^2+\sigma^2(r) h$ be the warped product metric. ${\bar\Gamma}_{ab}^c$ (respectively $\Gamma_{ij}^k$) the connection coefficients of $g$ and $h$ respectively.
We use a local coordinate system $\{r, u^i\}$ on $M^m_{\sigma}(o)$ where $\{ u^i\}$ is a local coordinate system on $S^{m-1}$. We will also use the following indices and their ranges: $a, b, c, d=0, 1, 2,\cdots, m-1$ $i, j, k, l=1, 2,\cdots, m-1$ for tensors on $(M^m_{\sigma}(o), g={\rm d}r^2+\sigma^2(r) h)$ and $(S^{m-1}, h)$ respectively. Then, it is easily checked that
\begin{align}\notag
&g_{00}=g^{00}=1,\; g_{0i}=g^{0i}=0,\;\; g_{ij}=\sigma^2(r)h_{ij},\;\; g^{ij}=\sigma^{-2}h^{ij}.
\end{align}

  By a straightforward  computation we have, for a function $f=f(x_0, x_1, \cdots, x_{m-1})$ with $(x_1, \cdots,x_{m-1})=(u_1, \cdots, u_{m-1})\in S^{m-1}$, 

\begin{align}\label{L1}
\Delta_{\sigma} f=&g^{ab}f_{ab}-g^{ab}{\bar\Gamma}_{ab}^cf_c=g^{00}f_{00} +g^{ij}f_{ij} -g^{00}{\bar\Gamma}_{00}^cf_c-g^{ij}{\bar\Gamma}_{ij}^cf_c\\\notag
=&g^{00}f_{00} +g^{ij}f_{ij} -g^{00}{\bar\Gamma}_{00}^0f_0 -g^{00}{\bar\Gamma}_{00}^k f_k-g^{ij}{\bar\Gamma}_{ij}^0f_0 -g^{ij}{\bar\Gamma}_{ij}^k f_k\\\notag
=&f_{rr} +g^{ij}f_{ij}  -{\bar\Gamma}_{00}^k f_k-g^{ij}{\bar\Gamma}_{ij}^0f_0 -g^{ij}{\bar\Gamma}_{ij}^k f_k\\\notag
=&f_{rr} +(m-1)\frac{\sigma'}{\sigma}\, f_r+ \frac{1}{\sigma^2}\,\Delta_{S^{m-1}} f
\end{align}
where in obtaining the last equality we have used the identities of the connection coefficients of  the warped product metrics 
\begin{align}\notag
{\bar\Gamma}_{ij}^0=-\sigma\sigma'h_{ij},\;\;{\bar\Gamma}_{ij}^k=\Gamma_{ij}^k.
\end{align}

\begin{equation}\notag
\begin{aligned}
\Delta_{\sigma}^2 f=&(\Delta_{\sigma} f)_{rr} +(m-1)\frac{\sigma'}{\sigma}\, (\Delta_{\sigma} f)_r+ \frac{1}{\sigma^2}\,\Delta_{S^{m-1}} (\Delta_{\sigma} f).
\end{aligned}
\end{equation}
For radial function $f(r, u)=f(r)$, we have $\Delta_{S^{m-1}}f=0$ and hence the formulas (\ref{2HF}) and (\ref{LMr}) follow.
\end{proof}

As a straightforward application we  obtain the following complete description of all radial biharmonic functions on the model space $(M^m_{\sigma}(o), g={\rm d}r^2+\sigma^2(r) g^{S^{m-1}})$.

\begin{theorem}\label{P1}
A radial function $f(r)$ on the model space \\$(\mathbb{R}^+\times S^{m-1}, {\rm d}r^2+\sigma^2(r)g^{S^{m-1}})$  is biharmonic if and only if
 \begin{align}\notag
 f(r)= &\; c_4+ c_3\int \frac{1}{\sigma^{m-1}}\,dr\\
 &+c_2\Big[\int \frac{1}{\sigma^{m-1}}\,dr\,\int \sigma^{m-1} dr-\int \left(\int \frac{1}{\sigma^{m-1}}\,dr\right)\sigma^{m-1}\,dr\Big]\\\notag
 &+c_1\Big[ \int \frac{1}{\sigma^{m-1}}\,dr\,\int \left(\int \frac{1}{\sigma^{m-1}}\,dr\right)\sigma^{m-1}\,dr-\int \left(\int \frac{1}{\sigma^{m-1}}\,dr\right)^2\sigma^{m-1}\,dr\Big].
 \end{align}
\end{theorem}
\begin{proof}
By  (\ref{LMr}), a radial  function $f(r, u)=f(r)$ is biharmonic on  $M^m_{\sigma}(o)$ if and only if $y=\Delta_{\sigma} f$ is a solution of the second order ordinary homogeneous differential equation
\begin{equation}\label{S1}
y''+(m-1)\frac{\sigma'}{\sigma}\, y'=0.
\end{equation}
 It is easy to check that the general solutions of (\ref{S1})  are 
\begin{align}\label{S10}
y=c_1y_1+c_2y_2,\;\; with\; y_1=\int\frac{1}{\sigma^{m-1}} dr, \;\;\;\;\;y_2=-1.
\end{align}
Thus, the biharmonicity of $f(r, u)=f(r)$ leads to the inhomogeneous 2nd order ODE
\begin{align}\label{S11}
\Delta f=c_1y_1+c_2y_2,\;\; with\; y_1=\int\frac{1}{\sigma^{m-1}} dr, \;\;\;\;\;y_2=-1,\; or\\
f''+(m-1)\frac{\sigma'}{\sigma}\, f'=c_1y_1+c_2y_2,\;\; with\; y_1=\int\frac{1}{\sigma^{m-1}} dr, \;\;\;\;\;y_2=-1.
\end{align}
Note that the homogeneous part of this equation is exactly (\ref{S1}) whose general solution is known to be (\ref{S10}).
By the theory of the general solutions of linear ordinary differential equations, we conclude that the general solutions of  (\ref{S11}) can be written as
\begin{align}
f=A y_1+B y_2+y_{p_1}+y_{p_2},
\end{align}
where $y_{p_1}$ and $y_{p_2}$ are a particular solution of

\begin{align}
y''+(m-1)\frac{\sigma'}{\sigma}\, y'= &\,c_1\int\frac{1}{\sigma^{m-1}} dr, \;and\\
y''+(m-1)\frac{\sigma'}{\sigma}\, y'=&\, c_2,
\end{align}
 respectively.
 By using the method of variation of parameters (see e.g., \cite{ZC}, \S 4.6) we find that
 \begin{align}
 y_{p_1}= &\,\;c_1\left( y_1\,\int y_1\sigma^{m-1}\,dr-\int y_1^2\sigma^{m-1}\,dr\right)\\
  y_{p_2}= &\,\;c_2\left(y_1\,\int \sigma^{m-1} dr-\int y_1\sigma^{m-1}\,dr\right).\\
 \end{align}
 Therefore, a radial  function $f(r)$ is biharmonic function on the model space  $(M^m_{\sigma}(o)=(\mathbb{R}^+\times S^{m-1}, {\rm d}r^2+\sigma^2(r)g^{S^{m-1}})$ if and only if
 \begin{align}
 f(r)=c_3y_1+c_4+c_1\left( y_1\,\int y_1\sigma^{m-1}\,dr-\int y_1^2\sigma^{m-1}\,dr\right)\\
 +c_2\left(y_1\,\int \sigma^{m-1} dr-\int y_1\sigma^{m-1}\,dr\right),
 \end{align}
 where $y_1=\int \frac{1}{\sigma^{m-1}}\,dr$. From this we obtain the integral formula for biharmonic functions on the model space. 
 \end{proof}

Applying Theorem \ref{P1} with $\sigma(r)= r, \;\sin r, \sinh r$ respectively we obtain the following descriptions of radial biharmonic functions on space forms.
 
 \begin{proposition}\label{CB}
  (i)  On the Euclidean space  $(\r^m\setminus\{0\}, dr^2+r^2  g^{S^{m-1}})$,  a  radial functions $f(r)$ is a biharmonic if and only if
 \begin{align}\label{ER}
 &f(r) =\begin{cases}
 c_4+c_3 \ln r+c_2 r^2+c_1r^2 \ln r,\;\;\; m=2,\\
 c_4+c_3 r^{-2}+c_2 r^2+c_1\ln r,\;\;\;\; m=4,\; and\\
 c_4+c_3r^{2-m}+c_2r^2+c_1r^{4-m},\;\;\;m\ne 2, 4.
 \end{cases}
 \end{align}
 
 (ii) On the sphere  $(S^m, dr^2+\sin^2 r g^{S^{m-1}})$, a radial function $f(r)$ is a biharmonic if and only if
 \begin{align}\label{FS}
 &f(r) =c_4+c_3\,\int \frac{1}{\sin^{m-1}r}\,dr\\\notag
& \; +c_2\Big[\int \frac{1}{\sin^{m-1}r}\,dr\,\int \sin^{m-1}r dr-\int \left(\int \frac{1}{\sin^{m-1}r}\,dr\right)\sin^{m-1}r\,dr\Big]\\\notag
 &+c_1\Big[ \int \frac{1}{\sin^{m-1}r}\,dr\,\int \left(\int \frac{1}{\sin^{m-1}r}\,dr\right)\sin^{m-1}r\,dr-\int \left(\int \frac{1}{\sin^{m-1}r}\,dr\right)^2\sin^{m-1}r\,dr\Big].
 \end{align}
 
  (iii) On the hyperbolic space  $(H^m, dr^2+\sinh^2 r g^{S^{m-1}})$, a radial function $f(r)$ is a biharmonic if and only if
 \begin{align}\label{FH}
&f(r) =c_4+c_3\,\int \frac{1}{\sinh^{m-1}r}\,dr\\\notag
& \; +c_2\Big[\int \frac{1}{\sinh^{m-1}r}\,dr\,\int \sinh^{m-1}r dr-\int \left(\int \frac{1}{\sinh^{m-1}r}\,dr\right)\sinh^{m-1}r\,dr\Big]\\\notag
 &+c_1\Big[ \int \frac{1}{\sinh^{m-1}r}\,dr\,\int \left(\int \frac{1}{\sinh^{m-1}r}\,dr\right)\sinh^{m-1}r\,dr-\int \left(\int \frac{1}{\sinh^{m-1}r}\,dr\right)^2\sinh^{m-1}r\,dr\Big].
 \end{align}
 \end{proposition}
 
 Note that for $m=3$ we obtain the following examples which was obtained in \cite{Ca}
 \begin{example}
(I)  A radial function $f(r)$ on 3-sphere $(S^3, dr^2+\sin^2 r g^{S^{2}})$ is biharmonic if and only if
 \begin{align}\label{S30}
f(r)=a_1+a_2\,\cot r+a_3\, r+a_4\,r \cot r,
\end{align}
where $a_i$ are arbitrary constants.\\
(II)  A radial function $f(r)$ on hyperbolic space $(H^3, dr^2+\sinh^2 r g^{S^{2}})$ is  biharmonic if and only if
 \begin{align}\label{H30}
f(r)=a_1+a_2\,\coth r+a_3\, r+a_4\,r \coth r,
\end{align}
where $a_i$ are arbitrary constants.\\

In Fact, using Proposition \ref{CB} with $m=3$ we know that  a radial function $f(r)$ on $(S^3, dr^2+\sin^2 r g^{S^{2}})$ is biharmonic if and only if
 \begin{align}
 &f(r) =c_4+c_3 \int \frac{1}{\sin^{2}r}\,dr\,\\\notag
 & c_2\Big[\int \frac{1}{\sin^{2}r}\,dr\,\int \sin^{2}r dr-\int \left(\int \frac{1}{\sin^{2}r}\,dr\right)\sin^{2}r\,dr\Big]\\\notag
 &+c_1\Big[ \int \frac{1}{\sin^{2}r}\,dr\,\int \left(\int \frac{1}{\sin^{2}r}\,dr\right)\sin^{2}r\,dr-\int \left(\int \frac{1}{\sin^{2}r}\,dr\right)^2\sin^{2}r\,dr\Big].
 \end{align}
A straightforward computation gives
\begin{align}
f(r) &= c_4-c_3 \cot r-\frac{c_2}{2} r \cot r-\frac{c_1}{2}(r+\cot r)\\
&=c_4-(c_3+\frac{c_1}{2}) \cot r-\frac{c_2}{2} r \cot r-\frac{c_1}{2}r,
\end{align}
from which we obtained Statement (I). A similar proof gives Statement (II)
\end{example}

 \begin{remark}
(i) Note that in the family (\ref{S30})  $a_1+a_2\,\cot r$ are harmonic and $a_3\, r+a_4\,r \cot r$ are proper biharmonic on $S^3$. More precisely, $a_4\,r \cot r$ is a family of qualsi-harmonic functions (i.e., $\Delta^{S^3} (a_4 r \cot r)=constant\ne 0$) and the function $a_3 r$ is proper biharmonic. Similarly,  in the family (\ref{H30})  $a_1+a_2\,\coth r$ are harmonic and $a_3\, r+a_4\,r \coth r$ are proper biharmonic on $H^3$. Furthermore, $a_4\,r \coth r$ is qualsi-harmonic  and $a_3 \r$ is proper biharmonic.\\
(ii) An interesting result proved in \cite{Ca} states that the distance function $d(r) =r$ is proper biharmonic in a Riemannian manifold $(M^m, g)$ if and only if $m=1$, or $m=3$ and the manifold is a space form. Our example above and Proposition \ref{CB} (i) confirm  that the function $u(r) =r$ in 3-dimensional space form is indeed proper biharmonic.
 \end{remark}
 
\begin{corollary}
  (A) A radial function $f(r)$ on the sphere $(S^2, dr^2+\sin^2 r d\theta^2)$ is biharmonic if and only if
  \begin{align}\label{q1}
 &f(r) =a_1+a_2\,\ln \tan \frac{r}{2} +a_3\,\ln \sin r \\\notag
 &+a_4 \Big[ \ln\sin r\ln\tan\, \frac{r}{2}-2\int \cot r\,\ln\tan  \frac{r}{2}\,dr\Big], 
  \end{align}
where $a_i$ are arbitrary constants. In particular, $f(r)=\ln \sin r$ is a proper biharmonic function.\\
(B) A radial function $f(r)$ on the hyperbolic space $(H^2, dr^2+\sinh^2 r d\theta^2)$ is biharmonic if and only if
  \begin{align}
 &f(r) =a_1+a_2\,\ln \tanh \frac{r}{2} +a_3\,\ln \sinh r \\\notag
 &+a_4 \Big[ \ln\sinh r\ln\tanh\, \frac{r}{2}-2\int \coth r\,\ln\tanh  \frac{r}{2}\,dr\Big], 
  \end{align}
where $a_i$ are arbitrary constants. In particular, $f(r)=\ln \sin r$ is a proper biharmonic function.\\
\end{corollary}
 \begin{proof} Applying formula (\ref{FS}) with $m=2$ we have
 \begin{align}
 &f(r) =c_4+c_3\,\int \frac{1}{\sin r}\,dr\\\notag
& \; +c_2\Big[\int \frac{1}{\sin r}\,dr\,\int \sin r dr-\int \left(\int \frac{1}{\sin r}\,dr\right)\sin r\,dr\Big]\\\notag
 &+c_1\Big[ \int \frac{1}{\sin r}\,dr\,\int \left(\int \frac{1}{\sin r}\,dr\right)\sin r\,dr-\int \left(\int \frac{1}{\sin r}\,dr\right)^2\sin r\,dr\Big].
 \end{align}
 A straightforward calculation yields
  \begin{align}
 &f(r) =(c_4+c_2\ln 2)+(c_3-c_1\ln 2)\,\ln \tan \frac{r}{2} -c_2\,\ln \sin r \\\notag
 &+c_1\Big[ \ln\sin r\ln\tan\, \frac{r}{2}-2 \int \cot r\,\ln\tan  \frac{r}{2}\,dr\Big].
 \end{align}
 Relabeling the arbitrary constants we obtain (\ref{q1}). A similar proof gives Statement (B).
\end{proof}

{\bf Radial biharmonic functions on space forms with conformally flat models:} One can easily check that by performing a coordinate change $r= 2\tan^{-1} t$ the warped product
 model for a sphere  $ (S^m, dr^2+\sin^2 r g^{S^{m-1}})$  can be written as the  conformally flat model $ \left(S^m, \frac{4(dt^2+t^2g^{S^{m-1}})}{(1+t^2)^2}\right)$. Similarly, the hyperbolic space $ (H^m, dr^2+\sinh^2 r g^{S^{m-1}})$ as can be written as $ \left(H^m, \frac{4(dt^2+t^2g^{S^{m-1}})}{(1-t^2)^2}\right)$. 
 
Using these models and the transformations $r= 2\tan^{-1} t$ and $r= 2\tanh^{-1} t$ respectively,  Proposition \ref{CB}  reads

\begin{proposition}\label{St}
(i) A radial function $u: \left(S^m, \frac{4(dt^2+t^2g_{S^{m-1}})}{(1+t^2)^2}\right)\to \r, u=u(t)$ on the sphere is biharmonic if and only if
\begin{align}\label{2S}
u(t)= \; & c_1+c_2 u_2+c_3\Big[u_2\int\frac{t^{m-1}}{(1+t^2)^m} dt-\int \frac{u_2\,t^{m-1}}{(1+t^2)^m} dt\Big]\\\notag
&\,+c_4\Big[u_2\int\frac{u_2\, t^{m-1}}{(1+t^2)^m} dr-\int \frac{u_2^2\,t^{m-1}}{(1+t^2)^m} dt\Big],
\end{align}
where $u_2=\int \frac{(1+t^2)^{m-2}}{t^{m-1}}\,dt$.\\

(ii) A radial function $u: \left(H^m, \frac{4(dt^2+t^2g_{S^{m-1}})}{(1-t^2)^2}\right)\to \r, u=u(t)$ on the hyperbolic space  is biharmonic if and only if
\begin{align}\label{2H}
u(t)= \; & c_1+c_2 u_2+c_3\Big[u_2\int\frac{t^{m-1}}{(1-t^2)^m} dr-\int \frac{u_2\,t^{m-1}}{(1-t^2)^m} dt\Big]\\\notag
&\,+c_4\Big[u_2\int\frac{u_2\, t^{m-1}}{(1-t^2)^m} dt-\int \frac{u_2^2\,t^{m-1}}{(1-t^2)^m} dr\Big],
\end{align}
where $u_2=\int \frac{(1-t^2)^{m-2}}{t^{m-1}}\,dr$.
\end{proposition}

A straightforward checking using  (\ref{2S}) and (\ref{2H}) with $m=2$ respectively we recover the following radial biharmonic functions on $S^2$ and $H^2$ which was obtained in \cite{Le}. 
\begin{corollary}(see also \cite{Le})
(i) A radial function $u: \left(S^2, \frac{4(dt^2+t^2d\theta^2)}{(1+t^2)^2}\right)\to \r$ with $u=u(t)$ is biharmonic if and only if
\begin{align}\label{2S2}
u(t)= \; & c_1+c_2 \ln t+c_3\ln (1+t^2)+c_4\Big[\ln t \ln (1+t^2)- 2\int \frac{\ln (1+t^2)}{t} dt\Big].
\end{align}
(ii) A radial function $u: \left(H^2, \frac{4(dt^2+t^2d\theta^2)}{(1-t^2)^2}\right)\to \r$ with $u=u(t)$ is biharmonic if and only if
\begin{align}\label{2H2}
u(t)= \; & c_1+c_2 \ln t+c_3\ln (1-t^2)+c_4\Big[\ln t \ln (1-t^2)- 2\int \frac{\ln (1-t^2)}{t} dt\Big].
\end{align}
\end{corollary}

\begin{example}
The function $u(t)=\arctan t $ is a proper biharmonic function on  $\left(S^3, \frac{4(dt^2+t^2g^{S^2})}{(1+t^2)^2}\right)\to \r$. Note that in this model, one can easily described the function by using the Cartesian coordinates from $\r^4$ in which $S^3$  lies. In fact,  this function can be  viewed as $\r^4\supset S^3\to \r$ and  $u (x)=\arctan \sqrt{\frac{1+x_4}{1-x_4}}$  for any $(x_1, x_2, x_3, x_4)\in S^3\setminus\{N\}$.
\end{example}

Note that radial quasi-harmonic functions on Poincar\'e  n-ball \\ $B^n_{\alpha}=(D^n, (1-|x|^2)^{2\alpha}g_0)$ were studied in \cite{SW1} and \cite{SW7}, where the authors proved that there exist bounded quasi-harmonic functions on $B^n_{\alpha}$ if and only if $\alpha\in (-1, \frac{1}{n-2})$.\\

\section{Biharmonic functions on model spaces via spherical harmonic functions}

In their study of the generalized Almansi property the authors in \cite{MR} explored the possibility of constructing biharmonic functions on the model space\\ $M^m_{\sigma}(o)=(\mathbb{R}^+\times S^{m-1}, {\rm d}r^2+\sigma^2(r)g^{S^{m-1}})$ via the products of a radial function and a local harmonic function. It turns out that the only model space that has this property is the Euclidean space. 

In this section, we will prove that the model space \M always admits proper biharmonic functions $f(r, \theta)=u(r)v_k(\theta)$  as a product of any eigenfunction on the factor sphere $S^{m-1}$ and certain radial functions.

 This provides an effective way to construct infinitely many proper biharmonic functions on the model space including the three space forms.
 
\begin{theorem}\label{MT2}
For any eigenfunction $v_k:S^{m-1}\to \r$ there exist radial functions $u(r)$ such that the product $p(r,\theta)=u(r)v_k(\theta)$ is a  biharmonic functions on the model space  $M^m_{\sigma}(o)=(\mathbb{R}^+\times S^{m-1}, {\rm d}r^2+\sigma^2(r)g^{S^{m-1}})$. In particular, there are infinitely many proper biharmonic functions on space forms given by the product of any  eigenfunctions with some radial functions.
\end{theorem}
\begin{proof}

Using polar coordinates $(r, \theta)\in \r^{+}\times S^{m-1}=\r^m\setminus\{0\}$ and (\ref{LM}), we have the Laplacian  on the function $f(r, \theta)=u(r)v_k(\theta)$ on  $M^m_{\sigma}(o)$ as
\begin{align}\label{LR}
\Delta f(r, \theta)=\big[ u'' +(m-1)\frac{\sigma'}{\sigma}\, u' \big] v_k(\theta)+ \frac{u}{\sigma^2}\,\Delta^{S^{m-1}} (v_k(\theta)).
\end{align}
Since $v_k(\theta)$ is a $\lambda_k$-eigenfunction on $S^{m-1}$ (i.e., a restriction of a harmonic homogeneous polynomial of degree $k$ in $\r^m$ to $S^{m-1}$), we have 
\begin{equation}
\Delta_{S^{m-1}} v_k(\theta)=-k(m+k-2) v_k(\theta)
\end{equation}
 and hence (\ref{LR}) becomes

\begin{align}\label{LR1}
\Delta f(r, \theta)=   \big[u'' +(m-1)\frac{\sigma'}{\sigma}\, u'- \frac{k(m+k-2)}{\sigma^2} u\big]v_k(\theta).
\end{align}
Thus, a function $f(r, \theta)=u(r)v_k(\theta)$ is harmonic on the model space $M^m_{\sigma}(o)$ if and only if the radial function $u(r)$ solves the 2nd order linear homogeneous equation
 
\begin{align}\label{MH}
u'' +(m-1)\frac{\sigma'}{\sigma}\, u'- \frac{k(m+k-2)}{\sigma^2} u=0.
\end{align}

By the well known existence theory of linear differential equations (see e.g., \cite{ZC}, Theorem 4.14), we have the general solution of (\ref{MH}) given by $u=c_1u_1+c_2 u_2$.

This means that  $f(r, \theta)=u(r)v_k(\theta)$ is harmonic if and only if $f(r, \theta)= (c_1u_1(r)+c_2 u_2(r))v_k (\theta)$. 

On the other hand, recall that $p(r, \theta)=u(r)v_k(\theta)$ is biharmonic on the model space if and only if $\Delta p(r, \theta)$ is  harmonic  on the model space. Looking for biharmonic function $p(r, \theta)=u(r)v_k(\theta)$ with $\Delta p(r, \theta)= (c_1u_1+c_2u_2)v_k(\theta)$ and using  (\ref{LR1}) we are lead to the following non-homogeneous  2nd order differential equation
\begin{align}\label{M2H}
u'' +(m-1)\frac{\sigma'}{\sigma}\, u'- \frac{k(m+k-2)}{\sigma^2} u=c_1u_1+c_2u_2.
\end{align}
 By the method of variation of parameters (see, e.g., \cite{ZC} \S 4.6) we have the general solution of  (\ref{M2H}
) given by
\begin{align}\label{GS}
u(r)=c_1u_1(r)+c_2u_2(r) + c_3u_{p_1}(r) +c_4(r)u_{p_2}(r),
\end{align}
where
\begin{align}\label{up1}
u_{p_1}(r)= &\; u_1\int\frac{-u_1u_2}{u_1u_2'-u_1'u_2}\,dr+u_2\int\frac{u_1^2}{u_1u_2'-u_1'u_2}\,dr,\;and\\\label{up2}
u_{p_2}(r)= &\; u_1\int\frac{-u_2^2}{u_1u_2'-u_1'u_2}\,dr+u_2\int\frac{u_1 u_2}{u_1u_2'-u_1'u_2}\,dr.
\end{align}

Summarizing the above, we conclude that  for any eigenfunction $v_k(\theta)$ on $S^{m-1}$ the product $p(r, \theta)= u(r)v_k (\theta)$ with $u(r)$ given by (\ref{GS}) is a biharmonic function on the model space  $(\mathbb{R}^+\times S^{m-1}, {\rm d}r^2+\sigma^2(r)g^{S^{m-1}})$. This completes the proof of the theorem.

\end{proof}
Theorem \ref{MT2}  (Formula (\ref{GS})) provides an effective method to construct biharmonic functions on model space including the three space forms.

\begin{corollary}\label{R2}
On the punctured Euclidean plane $(\r^2\setminus\{0\}, dr^2+ r^2 d\theta^2)$, each function in  the family 
\begin{align}\label{BF}
p(r,\theta)= ( c_3 + c_4 r^4) \;v_2(\theta)
\end{align}  
is a  proper biharmonic function. In particular, \\ $p(r,\theta)= a\, cos\, 2\theta +b\, \sin\, 2\theta=\frac{a(x^2-y^2)+b(xy)}{x^2+y^2}$ is a family of bounded  proper biharmonic functions on the punctured Euclidean plane.
\end{corollary}
\begin{proof}
Using Theorem \ref{MT2}  with $m=2, k=2,  \sigma(r)= r$ Equation (\ref{MH}) becomes
\begin{align}\label{S20}
u'' +\frac{1}{r}\, u'-\frac{4}{r^2} u=0
\end{align}
which has the fundamental set of solutions $u_1(r)=r^2, \;u_2(r)=r^{-2}$. Using formulas (\ref{up1}) and (\ref{up2}) we have
\begin{align}\notag
u_{p_1}(r)=  \frac{r^4}{12}, \;\;u_{p_2}(r)=  -\frac{1}{4}.
\end{align}
Thus, we obtain the family (\ref{BF}) of proper biharmonic functions. Recalling that $v_2(\theta)$ is the restriction of a harmonic homogeneous polynomial of degree $2$, so we have
\begin{align}
v_2(\theta)=[a(x^2-y^2)+b\, xy]|_{S^1}=\frac{a(x^2-y^2)+b(xy)}{x^2+y^2}=a\, cos\, 2\theta +b\, \sin\, 2\theta,
\end{align}
which is clearly a bounded  function on the puncture plane $\r^2\setminus\{0\}$.
\end{proof}

\begin{remark}
(i) The interesting problem of classifying all bounded biharmonic functions on a punctured Euclidean space was studied in \cite{SW0}, where the authors obtained the  following nice classification.\\
{\bf Theorem} (\cite{SW0}): The linear space $BH^2(\r^m\setminus\{0\})$ of bounded biharmonic functions on $\r^m\setminus\{0\}$ is 
\begin{align}
BH^2(\r^m\setminus\{0\})=\begin{cases} span \{ 1, \cos 2\theta, \sin 2\theta\},\hskip2.5cm for\; m=2;\\
span \{ 1, \cos r\cos \theta, \sin r \sin \theta, \cos r\}\;\;\; for\; m=3;\\
Span \{ 1 \}, \hskip4.8cm for \;m\ge 4.
\end{cases}
\end{align}
(ii) Our Corollary \ref{R2} does confirm that the functions $\cos 2\theta, \sin 2\theta $ are indeed bounded proper biharmonic on the punctured plane.
 \end{remark}

\begin{example}
On the sphere $(S^2\setminus\{\pm N\}, dr^2+ \sin^2r d\theta^2)$, each function in  the family 
\begin{align}\notag
p(r,\theta)= \big[\;& c_3(-\sin^2\frac{r}{2}(\cot \frac{r}{2}+\tan \frac{r}{2})+2 \tan \frac{r}{2}\ln \sin \frac{r}{2})\\
\; +&\; c_4(\sin^2\frac{r}{2}(\cot \frac{r}{2}+ \tan \frac{r}{2})+2 \cot \frac{r}{2}\ln \cos \frac{r}{2})\big] \;v_1(\theta)
\end{align}  
is a  proper biharmonic function. In particular, \\$p(r,\theta)= (2 \tan \frac{r}{2}\ln \sin \frac{r}{2}
\; +2 \cot \frac{r}{2}\ln \cos \frac{r}{2}) v_1(\theta)$ is a proper biharmonic function on 2-sphere.

In fact,  using Theorem \ref{MT2}  with $m=2, k=1, \sigma(r)= \sin r$, Equation (\ref{MH}) becomes 
\begin{align}\label{S20}
u'' +\cot r\, u'-\csc^2r u=0
\end{align}
which has the general solution $u(r)=c_1u_1(r)+c_2u_2(r)$ for $u_1=\cot \frac{r}{2}, u_2=\tan\frac{r}{2}$. Using formulas (\ref{up1}) and (\ref{up2}) we have
\begin{align}\notag
u_{p_1}(r)= &\;\;\;\;\; \frac{1}{2}\cot \frac{r}{2} -\sin^2\frac{r}{2}(\cot \frac{r}{2}+\tan \frac{r}{2})+2 \tan \frac{r}{2}\ln \sin \frac{r}{2}, \\\notag
u_{p_2}(r)= &\; -\frac{1}{2}\tan \frac{r}{2} +\sin^2\frac{r}{2}(\cot \frac{r}{2}+ \tan \frac{r}{2})+2 \cot \frac{r}{2}\ln \cos \frac{r}{2}.
\end{align}
Noting that the first term in both $u_{p_1}(r)$ and $u_{p_2}(r)$ above are  harmonic, by discarding them from the families we obtain the claimed family of proper biharmonic functions. The particular one is obtained by choosing $c_3=c_4=1$.
\end{example}

Similarly, by using Theorem \ref{MT2}  with $m=2, k=1,  \sigma(r)= \sinh r$ we have 
\begin{example}
On the hyperbolic plane $(H^2, dr^2+ \sinh^2r d\theta^2)$, each function in  the family 
\begin{align}\notag
p(r,\theta)= \big[\;& c_3(-\sinh^2\frac{r}{2}(\coth \frac{r}{2}-\tanh \frac{r}{2})+2 \tanh \frac{r}{2}\ln \sinh \frac{r}{2})\\
\; +&\; c_4(-\sinh^2\frac{r}{2}(\coth \frac{r}{2}- \tanh \frac{r}{2})+2 \coth \frac{r}{2}\ln \cosh \frac{r}{2})\big] \;v_1(\theta)
\end{align}  
is a  proper biharmonic function. In particular, \\$p(r,\theta)= (2 \tanh \frac{r}{2}\ln \sinh \frac{r}{2}
\; -2 \coth \frac{r}{2}\ln \cosh \frac{r}{2}) v_1(\theta)$ is a proper biharmonic function on the hyperbolic plane $H^2$.\\
\end{example}

\begin{example}
On the sphere $(S^2\setminus\{\pm N\}, dr^2+ \sin^2r d\theta^2)$, each function in  the family 
\begin{align}\notag
p(r,\theta)= \big[\;& c_3(1+\cos^2 \frac{r}{2} -\sin^2\frac{r}{2}\tan^2 \frac{r}{2}+4 \tan^2 \frac{r}{2}\ln \sin \frac{r}{2})\\
\; +&\; c_4( \cos^2 \frac{r}{2} -\sin^2\frac{r}{2}\tan^2 \frac{r}{2}+2 \cot^2 \frac{r}{2}\ln \cos \frac{r}{2})\big] \;v_2(\theta)
\end{align}  
is a  proper biharmonic function. In particular, \\$p(r,\theta)= (1+4 \tan^2 \frac{r}{2}\ln \sin \frac{r}{2}
\; -2 \cot^2 \frac{r}{2}\ln \cos \frac{r}{2}) v_2(\theta)$ is a proper biharmonic function on 2-sphere.

In fact,  with $m=k=2, \sigma(r)=\sin r$ Equation (\ref{MH}) becomes
\begin{align}\label{S20}
u'' +\cot r\, u'-4 \csc^2r u=0
\end{align}
which has a set of fundamental solutions  $u_1=\cot^2 \frac{r}{2}, u_2=\tan^2\frac{r}{2}$. Using formulas (\ref{up1}) and (\ref{up2}) and discarding the harmonic parts we have
\begin{align}\notag
u_{p_1}(r)= &\;\;\;\;\; 1+\cos^2 \frac{r}{2} -\sin^2\frac{r}{2}\tan^2 \frac{r}{2}+4 \tan^2 \frac{r}{2}\ln \sin \frac{r}{2}, \\\notag
u_{p_2}(r)= &\hskip1.2cm  \cos^2 \frac{r}{2} -\sin^2\frac{r}{2}\tan^2 \frac{r}{2}+2 \cot^2 \frac{r}{2}\ln \cos \frac{r}{2},
\end{align}
which gives the claimed family of proper biharmonic functions on $S^2$. The particular one is obtained by choosing $c_3=-c_4=1$.
\end{example}

\section{ Biharmonic functions on punctured Euclidean spaces}

For harmonic functions on a Euclidean space we have the following well known Liouville type theorems.\\
\indent {\bf Liouville Theorem}: Any bounded harmonic function on $\r^m$  is constant. \\
\indent {\bf Generalization 1}: Any positive harmonic function on $\r^m$  is constant. \\
\indent {\bf Generalization 2} (see e.g., \cite{ABR}, Corollaries 3.3, 3.14): Any positive harmonic function on $u:\r^m\setminus\{0\}\to \r$ takes the form
\begin{align}\label{PH}
u(x)=\begin{cases}
constant,\;\; m=2\\
a+b|x|^{-(m-2)},\;for\; a, b > 0,\;\;\;m>2.
\end{cases}
\end{align}

For biharmonic functions on a Euclidean space we have\\
\indent {\bf Liouville Theorem} (\cite{Hu}): Any bounded biharmonic function on $\r^m$  is constant. \\
\indent {\bf Fact}: There are positive biharmonic function on $\r^m$ which is NOT constant,\\
\indent  e.g., $u(x)=1+|x|^2$. \\
\indent For bounded biharmonic functions on the punctured Euclidean space we have\\
{\bf Theorem} (\cite{SW0}): The linear space $BH^2(\r^m\setminus\{0\})$ of bounded biharmonic functions on $\r^m\setminus\{0\}$ is 
\begin{align}\notag
BH^2(\r^m\setminus\{0\})=\begin{cases} span \{ 1, \cos 2\theta, \sin 2\theta\},\hskip2.5cm for\; m=2;\\
span \{ 1, \cos r\cos \theta, \sin r \sin \theta, \cos r\}\;\;\; for\; m=3;\\
Span \{ 1 \}, \hskip4.8cm for \;m\ge 4.
\end{cases}
\end{align}
In particular, any bounded biharmonic function on $\r^m\setminus\{0\}$ with $m\ge 4$ is constant.

Regarding biharmonic functions on $u:\r^m\setminus\{0\}\to \r$ we also have

{\bf Almansi Theorem} (1899): Any biharmonic function $u$ on $\r^m\setminus\{0\}$ (or more generally, a star-shaped region) takes the form
\begin{equation}\label{Rm}
u=h_1+|x|^2 h_2,
\end{equation}
where $h_1, h_2$ are two harmonic functions.

A little more detailed description is the following

{\bf Lemma A (\cite{SWb}, Lemma 2.2):} Any biharmonic function $u$ on $\r^m\setminus\{0\}$ takes the form
\begin{equation}\label{Rm2}
u=\begin{cases}
h_1+|x|^2 h_2 + (ax_1+bx_2) \ln |x|, \hskip1cm for\; m=2;\\
h_1+|x|^2 h_2 + a \ln |x|, \hskip2.8cm for\; m=4;\\
h_1+|x|^2 h_2,\hskip4.4cm for\; m\ne 2,4.\\
\end{cases}
\end{equation}
where $a, b$ are arbitrary constants and $h_1, h_2$ are two arbitrary harmonic functions.\\

Note that both formulas for biharmonic functions  given in (\ref{Rm}) and (\ref{Rm2}) depend on two arbitrary harmonic functions. To the author's knowledge, no further classification has been found in the literature. \\

Here, we gives  a complete classification of biharmonic functions with condition $\Delta u>0$ on the punctured Euclidean space.
\begin{theorem}\label{MT3}
A function $u:\r^m\setminus\{0\}\to \r$ is biharmonic with $\Delta u >0$ if and only if
\begin{align}\label{NC}
u(x)=\begin{cases}
c_1+c_2 \ln |x|+a |x|^2,\;\hskip3cm for\; m=2;\\
c_1+c_2  |x|^{-2} +a |x|^2+b\ln |x|,\;\hskip1.5cm for\; m=4;\\
c_1+c_2|x|^{-(m-2)}+ a |x|^2+b |x|^{-(m-4)},\; for\; m\ne 2, 4,
\end{cases}
\end{align}
where $c_i$ are arbitrary constants,  for $m= 2, 3, 4$, $a, b\ge 0, a^2+b^2\ne 0$, and  for $m\ne 2, 3, 4$, the constants $a  \ge 0, \;b\le 0, a^2+b^2\ne 0$.
\end{theorem}

\begin{proof}
Note that $u:\r^m\setminus\{0\}\to \r$ is biharmonic with $\Delta u >0$ means that $\Delta u$  is a positive harmonic  function on $\r^m\setminus\{0\}$. So, by (\ref{PH}), we have
\begin{align}\label{Z}
\Delta u=\begin{cases}
C,\;\; m=2\\
a+b|x|^{-(m-2)},\;for\; a, b > 0,\;\;\;m>2.
\end{cases}
\end{align}
Using polar coordinates $x=(r, \theta)\in \r^+\times S^{m-1}\equiv \r^{m}\setminus\{0\}$ with $|x|=r, \theta=\frac{x}{|x|}\in S^{m-1}$. Then it is known (see e.g., \cite{Sh}, $\S$ 22.3) that  for any smooth function $u:\r^m\setminus\{0\}\to \r$ is a limit of a sequence $f_n=\sum_{i=1}^{t_n} a_{ni}(r) b_{ni}(\theta)$, where $b_{ni}(\theta)\in C^{\infty}(S^{m-1})$. Using the well-known fact that $C^{\infty}(S^{m-1})\subset L^2(S^{m-1})=\oplus_{i=0}^{\infty} \mathcal{H}^i$ and the bases $\{v_{k1}, v_{k2},\cdots v_{kn_k}\}$ for $\mathcal{H}^k$, we can expressed $f_n$ as
\begin{equation}
f_n=\sum_{k=0}^{\infty} \sum_{i=1}^{n_k}\,u^n_{ki}(r)v_{ki}(\theta).
\end{equation}
Since $u$ and $f_n$ are closed in $C^{\infty}$ topology for $n\ge N$, we conclude that $u$ is biharmonic function with $\Delta u >0$, then so is $f_n$ a biharmonic function with $\Delta f_n >0$ for $n\ge N$. 

Using (\ref{2HF}) with $\sigma=r$ and the fact that $\Delta^{S^{m-1}}\,(u_{ki}(r)v_{ki}(\theta))=-k(m+k-2) u_{ki}\,v_{ki}$, we have
\begin{align}
\Delta f_n=\sum_{k=0}^{\infty} \sum_{i=1}^{n_k}[(u^n_{ki})''+(m-1)r^{-1}(u^n_{ki})'-k(m+k-2)r^{-2}u^n_{ki}] \,v_{ki}.
\end{align}
It  follows that for any $n$, $f_n$ is a solution of (\ref{Z}) if and only if

\begin{align}\label{Z1}
(u^n_{ki})''+(m-1)r^{-1}(u^n_{ki})'-k(m+k-2)r^{-2}u^n_{ki}=0, \;\forall\; k>0,  i,\;and
\end{align}
\begin{align}\notag
& (u^n_{01})''+(m-1)r^{-1}(u^n_{01})'-k(m+k-2)r^{-2}u^n_{01}\\\label{Z2}=&\;\begin{cases}
C,\;\; m=2\\
a+b|x|^{-(m-2)},\;for\; a, b > 0,\;\;\;m>2.
\end{cases}
\end{align}
One can easily check that the general solution of Equation (\ref{Z1}) is $u^n_{ki}(r)= c_1r^k+c_2r^{-(m+k-2)}$ which means $u^n_{ki}(r)v_{ki}$ is a harmonic function for all $k>0, i$.
By solving Equation (\ref{Z2}) with condition $\Delta u^n_{01}>0$ we have the general solutions
\begin{align}\label{Z3}
v^n_{01}(x)=\begin{cases}
c_1+c_2 \ln |x|+a |x|^2,\;\hskip3cm for\; m=2,\\
c_1+c_2  |x|^{-2} +a |x|^2+b\ln |x|,\;\hskip1.5cm for\; m=4,\\
c_1+c_2|x|^{-(m-2)}+ a |x|^2+b |x|^{-(m-4)},\; for\; m\ne 2, 4,
\end{cases}
\end{align}
where for $m=2$, the constant $a>0$, for $m=3, 4$, the constants $a, b \ge 0, a^2+b^2\ne 0$, and for $m\ne 2, 3, 4$, the constants $a  \ge 0, \;b\le 0, a^2+b^2\ne 0$.\\

Summarizing the above we obtain that $f_n=\sum_{k=0}^{\infty} \sum_{i=1}^{n_k}\,u^n_{ki}(r)v^n_{ki}(\theta)$ is a biharmonic function on $\r^m\setminus\{0\}$ with $\Delta f_n >0$ if and only if $f_n=v^n_{01}$ given by  (\ref{Z3}), which does not depends on $n$. Since  $u (r, \theta)=\lim_{n\to \infty} f_n=v^n_{01}$ , we obtain the theorem. 

\end{proof}

\begin{remark}
(i) We can check that by opposite sign choices for the constants $a, b$ in Theorem \ref{MT3}, we obtain all biharmonic functions on the punctured Euclidean space with negative Laplacian, i.e., $\Delta u<0$.\\
(ii) Note also that there are many biharmonic functions  on $\r^m\setminus\{0\}$ whose Laplacian  do not have a fixed sign. For example, $u=r^2\ln r$ is biharmonic on $\r^2\setminus\{0\}$ with $\Delta u=4(1+\ln r)$ which can be any real number.
\end{remark}

\end{document}